\documentclass[12pt]{amsart}
\UseRawInputEncoding
\usepackage{amssymb}
\usepackage{cases}
\usepackage{bbm}
\usepackage{xcolor}
\usepackage{mathrsfs}
\usepackage{graphicx}
\usepackage{amsfonts}
\usepackage{enumerate}
\usepackage{marginnote}
\usepackage{enumitem} 
\usepackage{blindtext} 
\allowdisplaybreaks[3]
\usepackage{thm-restate}
\usepackage{comment}
\usepackage{amsmath, amsthm}
\usepackage{amssymb,  amsmath,  amsfonts,  latexsym}
\usepackage{enumerate}
\usepackage{mathtools}
\allowdisplaybreaks[3]

\newtheorem{thm}{Theorem}[section]

\newtheorem{lem}[thm]{Lemma}
\newtheorem{prop}[thm]{Proposition}
\newtheorem{example}[thm]{Example}
\newtheorem{remarks}[thm]{Remark}
\newtheorem{defn}[thm]{Definition}
\newtheorem{hyp}[thm]{Hypothesis}
\newtheorem{assumption}[thm]{Assumption}
\numberwithin{equation}{section}
\usepackage{hyperref}
\hypersetup{
    colorlinks=true,   
    linkcolor=blue,    
    citecolor=blue,    
    urlcolor=blue      
}

\date{}

\def\<{\langle}
\def\>{\rangle}

\def\d"{^{\prime\prime}}

\def\bhyp{\begin{hyp}}
\def\nhyp{\end{hyp}}

\def\beq{\begin{equation}}
\def\nneq{\end{equation}}

\def\bdef{\begin{defn}}
\def\ndef{\end{defn}}

\def\bthm{\begin{thm}}
\def\nthm{\end{thm}}

\def\bprop{\begin{prop}}
\def\nprop{\end{prop}}

\def\brmk{\begin{remarks}}
\def\nrmk{\end{remarks}}

\def\bexa{\begin{example}}
\def\nexa{\end{example}}

\def\blem{\begin{lem}}
\def\nlem{\end{lem}}

\def\bcor{\begin{cor}}
\def\ncor{\end{cor}}

\def\bexe{\begin{exe}}
\def\nexe{\end{exe}}

\def\bprf{\begin{proof}}
\def\nprf{\end{proof}}

\def\bdes{\begin{description}}
\def\ndes{\end{description}}

\title[Long term convergence rate of  S-K approximation by Stein's method]
{Long term convergence rate of  Smoluchowski-Kramers approximation by Stein's method}
\author{Shiyu Liu}
\address{Shiyu Liu,  School of Mathematics and Statistics,  Wuhan University,  Wuhan,  Hubei 430072,  PR China.}
\thanks{The first author and the second author are supported by NSFC  12471143.}
\email{syliu\_math@whu.edu.cn}

\author{Wei Liu}
\address{Wei Liu,  School of Mathematics and Statistics,  Wuhan University,  Wuhan,  Hubei 430072,  PR China.}
\email{wliu.math@whu.edu.cn}
\date{}

\author{Lihu Xu}
\address{Lihu Xu, 1. Department of Statistics and Probability, Michigan State University, East Lansing, MI 48824, USA； 2. Department of Mathematics,  Faculty of Science and Technology,  University of Macau,  Av. Padre
	Tom\'as Pereira,  Taipa,  Macau,  China}
\email{xulihu@msu.edu, lihuxu@um.edu.mo}

\begin{document}

\begin{abstract} 
	We consider the following second-order stochastic differential equation on $\mathbb{R}^{2d}$:
	\begin{equation*}
		\begin{cases}
			dX_t^m=Y_t^mdt, \\
			mdY_t^m=b(X_t^m)dt+\sigma(X_t^m)dB_t-Y^m_tdt, 
		\end{cases}
	\end{equation*}
where  $X^m_t$ and $Y^m_t$  represent the position and velocity of a particle at time $t$, $m>0$ denotes its mass, $b:\mathbb{R}^d \rightarrow \mathbb{R}^d$ is the drift field, $\sigma:\mathbb{R}^d \rightarrow \mathbb{R}^{d \times d}$ is the diffusion coefficient, and $\{B_t\}_{t \ge 0}$ is a $d$-dimensional standard Brownian motion. 

The Smoluchowski--Kramers approximation states that as 
$m \rightarrow 0$, this system converges to the limiting equation: 
\begin{equation*}
			dX_t=b(X_t)dt+\sigma(X_t)dB_t.
	\end{equation*}
Utilizing Stein's method, we prove that the $1$-Wasserstein distance between the invariant distribution of $X_t^m$ and that of its small-mass limit $X_t$ is of order $O(\sqrt{m}|\ln m|).$  Particularly, in the one-dimensional case, the convergence rate can be improved to $O(\sqrt{m}).$
\end{abstract}
\maketitle



\section{Introduction}\label{se1}
\subsection{Background}

Consider a particle of mass $m$ moving in a force field $b(x)+\sigma(x)\dot{W}_t$ subject to velocity-proportional friction.   Its dynamics are governed by the following system of kinetic stochastic differential equations  (SDEs in short):
\begin{equation}\label{EQ1}
	\begin{cases}
		dX_t^m=Y_t^mdt, \\
		mdY_t^m=b(X_t^m)dt+\sigma(X_t^m)dB_t-Y^m_tdt, 
	\end{cases}
\end{equation}
where $b:\mathbb{R}^d\to\mathbb{R}^d, \sigma:\mathbb{R}^d\to\mathbb{R}^{d\times d}$,  $X_t^m$ and $Y_t^m$ denote the particle's position and velocity, respectively, while $B_t$ is a $d-$dimensional Brownian motion. This type of kinetic stochastic dynamics has been extensively studied in recent years, see for instance \cite{EGZ2019,GLWZ,GMZ2021,GWXZ2026,GuWa2012} and the references therein. 
When the mass $m$ is sufficiently small, the system can be  approximated by the following SDE:
\begin{equation}\label{EQ2}
dX_t=b(X_t)dt+\sigma(X_t)dB_t, 
\end{equation}
which is knowed as the Smoluchowski-Kramers approximation (or the small mass limit). 

The Smoluchowski-Kramers approximation has been extensively studied in the literature. Nelson \cite{N} proved that when $b$ is Lipschitz continuous and $\sigma$ is constant,  $X_t^m$ converges to $X_t$  almost surely, uniformly on  compact time intervals. Subsequently, Freidlin et al. \cite{F}  showed that for Lipschitz coefficients  $b$  and $\sigma,$  for any $T>0$ and $\delta>0, $
\begin{align*}
	\lim\limits_{m\rightarrow0}\mathbb{P}(\sup_{t\in[0, T]}|X_t^m-X_t|>\delta)=0.
\end{align*}

A substantial body of work has been devoted to estimating the convergence rate of this approximation. In particular, \cite{DNB, XY} examined convergence under non-Lipschitz coefficients, while \cite{fbm1,fbm3,fbm2} extended the framework to systems driven by fractional Brownian motion. The asymptotic behavior of the velocity component in \eqref{EQ1} has been analyzed in \cite{NK1,NK2,NK3}. Quantitative convergence rates have been obtained in various metrics: 
Berry--Esseen bounds were derived in \cite{BRJ}, and convergence in total variation distance was studied in \cite{BRJ1}. For applications of the approximation method, we refer to \cite{C,HDC,HMVW}. Significant results have also been achieved for infinite-dimensional stochastic partial differential equations; see \cite{SMF,JFA,LDP1,SPDE4,SPDE6,SPDE5}.

While most existing results focus on finite-time intervals, Cerrai et al. \cite{SMF, JFA} extended the analysis to long-time behavior in the setting of stochastic partial differential equations, proving convergence in distribution without obtaining an explicit rate.

In this paper, we investigate the stationary solutions of systems \eqref{EQ1} and \eqref{EQ2}. Our primary contributions are summarized as follows:
\begin{itemize}[leftmargin=*]
\item Under the assumptions stated below, both equations \eqref{EQ1} and \eqref{EQ2} admit unique invariant measures. Let $\nu_m$ denote the marginal distribution (with respect to the position variable) of the invariant measure of the kinetic process $\{(X_t^m,Y_t^m)\}_{t\geq 0}$, and let $\nu$ be the invariant measure of the limiting diffusion $\{X_t\}_{t\geq 0}$. We prove that: 
$$\mathbb{W}_1(\nu_m,\nu) \le C \sqrt{m} |\log m|, \ \ \ \ d>1,$$
$$\mathbb{W}_1(\nu_m,\nu) \le C \sqrt{m},\ \ \ \ d=1,$$
where $\mathbb{W}_1$ denotes the 1-Wasserstein distance. To the best of our knowledge, there are currently no existing results quantifying the convergence rate between the invariant measures of the second order equation \eqref{EQ1} and its limiting equation \eqref{EQ2} as $m \rightarrow 0$.
\item To derive the convergence rates above, we generalize the Stein's method framework developed in \cite{FSX}. This generalization is robust and can be extended to similar approximation problems where Brownian motion is replaced with stable noise. The distinction between the one-dimensional and general $d$-dimensional cases arises from the differing regularity properties of the associated Stein equations in these settings.
\end{itemize}

Stein's method was originally introduced in \cite{Stein1972} and has since evolved into a powerful technique for probability approximations. It has been widely applied across probability theory, as evidenced by \cite{stein0, DeTe2023, FaKo2021,stein1, stein2, stein4}.


The remainder of this paper is organized as follows. We conclude Section~\ref{se1} by introducing the necessary notation and stating the main theorems. Section~\ref{se3} presents the auxiliary lemmas, while Section~\ref{se4} provides the proofs of the main theorems. Finally, the technical proofs for the results in Section~\ref{se3} are distributed as follows: Section~\ref{seerg} establishes Propositions~\ref{prop2} and~\ref{egeq2}; Section~\ref{se5} proves Lemmas~\ref{XMYMesti} and \ref{MOMX}; and Section~\ref{se66} demonstrates Propositions~\ref{regularity} and~\ref{logf}.

\subsection{Notation}
The inner product of verctors $x, y\in\mathbb{R}^d$ is denoted by $\langle x, y\rangle$ and {the Eulidean norm of $x \in \mathbb{R}^d$  by $|x|$}. For matrices $A,B\in\mathbb{R}^{d\times d},$ their Hilbert-Schmidt inner product is $\langle A, B\rangle_{\rm HS}:={\rm Tr}(A^TB)$  where $^\top$ denotes the transpose operation, and the associated norm is defined by $||A||^2:=\langle A,A\rangle_{\rm HS}$. 

Let $C_c^\infty(\mathbb{R}^d, \mathbb{R})$ denote the set of smooth functions from $\mathbb{R}^d$ to $\mathbb{R}$ with compact support. 
The class of all continuous functions from  $\mathbb{R}^d$ to $\mathbb{R}$ is denoted by $C(\mathbb{R}^d, \mathbb{R})$ , and for an integer $k \geq 1$, $C^k(\mathbb{R}^d, \mathbb{R})$ denotes the space of functions whose derivatives up to order $k$ are continuous.

For $f \in C^2(\mathbb{R}^d, \mathbb{R})$ and $u, u_1, u_2, x \in \mathbb{R}^d$, the directional derivative $\nabla_u f(x)$ and $\nabla_{u_2} \nabla_{u_1} f(x)$ are respectively defined by
\begin{align*}
\nabla_u f(x) &= \lim_{\epsilon \to 0} \frac{f(x + \epsilon u) - f(x)}{\epsilon}, \\
\nabla_{u_2} \nabla_{u_1} f(x) &= \lim_{\epsilon \to 0} \frac{\nabla_{u_1} f(x + \epsilon u_2) - \nabla_{u_1} f(x)}{\epsilon}.
\end{align*}
As usual, $\nabla f(x) \in \mathbb{R}^d$, $\nabla^2 f(x) \in \mathbb{R}^{d \times d}$, and $\Delta f(x) \in \mathbb{R}$ denote the gradient, the Hessian, and the Laplacian of $f$, respectively. It is known that $\nabla_u f(x) = \langle \nabla f(x), u \rangle$ and $\nabla_{u_2} \nabla_{u_1} f(x) = \langle \nabla^2 f(x), u_1 u_2^\top \rangle_\text{HS}$. 

We write $C_p(\mathbb{R}^d, \mathbb{R})$ for the set of all continuous functions with polynomial growth, and $C^2_p(\mathbb{R}^d, \mathbb{R})$ for those functions in $C^2(\mathbb{R}^d, \mathbb{R})$ whose value, first derivative, and second derivative all have polynomial growth.


Next we introduce the Wasserstein distance used in this paper. Let $\mathcal{P}$ denote the set of  probability measures on $\mathbb{R}^d$ with finite first-order moment. For a given random variable $X$, we write $\mathcal{L}(X)$ for its law.  
\begin{defn}[1-Wasserstein distance] For any $\mu, \nu\in\mathcal{P}$, the 1-Wasserstein distance is defined as 
	\begin{equation*}
		\mathbb{W}_1(\mu, \nu):=\inf\limits_{\pi\in\Pi(\mu, \nu)}\int_{\mathbb{R}^d\times\mathbb{R}^d}|x-y|\pi(dx, dy),  
	\end{equation*}
	where $\Pi(\mu, \nu)$ is the set of all couplings of $\mu$ and $\nu$, that is, probability measures on $\mathbb{R}^d\times\mathbb{R}^d$ with marginals $\mu$ and $\nu.$
\end{defn}
By Kantorovich duality, the $1-$Wasserstein distance admits the dual representation:
\begin{equation}
	\mathbb{W}_1(\mu, \nu)=\sup\limits_{h\in{\rm Lip}_0(1)}\left\{\int_{\mathbb{R}^d}h(x)\mu(dx)-\int_{\mathbb{R}^d}h(x)\nu(dx)\right\}, 
\end{equation}
where ${\rm Lip}_0(1):=\left\{h:\mathbb{R}^d\rightarrow\mathbb{R}||h(x)-h(y)|\leq|x-y| \text{ and } h(0)=0\right\}.$

\subsection{Assumptions and main results}


We make the following assumptions:
\begin{assumption}\label{A0}
	  ({\bf Lipschitz coefficients}) There exists a  constant $L>0$ such that for any $x,y\in\mathbb{R}^d,$
    \begin{align*}
    |b(x)-b(y)|+||\sigma(x)-\sigma(y)||\leq L|x-y|.
    \end{align*}
    This Lipschitz assumption implies, in particular, that the drift has at most linear growth: there exists a constant $L_b>0$ such that for every $x\in\mathbb{R}^d,$ $$|b(x)|^2\leq L_b(1+|x|^2).$$
\end{assumption}

\begin{assumption}\label{A1}
	({\bf Uniformly bounded ellipticity}) there exists a constant $\sigma_0>0$ such that for any $x\in\mathbb{R}^d,$
 \begin{align*}
     \sigma(x)\sigma^T(x)\geq\sigma_0I_d,
 \end{align*}
 where $I_d$ denotes the $d$-dimensional identity matrix. The uniform matrix norm of $\sigma$ is finite, i.e. $$||\sigma||_\infty:=\sup\limits_{x\in\mathbb{R}^d}||\sigma(x)||<\infty.$$
 \end{assumption}

\begin{assumption}\label{A2}
 ({\bf Dissipative condition}) There exist two positive constants $c_1,c_2$ such that  for any $x\in\mathbb{R}^d, $
	\begin{align*}
		\langle x, &b(x)\rangle\leq c_1-c_2|x|^2.
\end{align*}
\end{assumption}


Under the preceding Assumptions \ref{A0}, \ref{A1}, and \ref{A2}, both SDEs \eqref{EQ1} and \eqref{EQ2} are ergodic and admit unique invariant probability measures, denoted respectively by $\pi_m$ and $\nu$ (see Propositions \ref{prop2} and \ref{egeq2} below).

Let $\nu_m(dx)=\pi_m(dx,\mathbb{R}^d)$ be the marginal distribution of $\pi_m$ with respect to he position variable $x$. Our main results as follows:
\begin{thm}\label{thm1}
	Let Assumptions \ref{A0},  \ref{A1} and \ref{A2} hold. Then there exists a constant $C>0$ such that for any $m\leq\min\left\{\frac{c_2}{2L_b}, \frac{2}{c_2}, \frac{1}{e}\right\}, $
	\begin{align*}
		\mathbb{W}_1(\nu_m, \nu)\leq C\sqrt{m}|\ln m|.
	\end{align*}
\end{thm}

In the one-dimension case  the convergence rate can be improved.

\begin{thm}\label{thm2}
Let $d=1$, and suppose Assumptions \ref{A0},  \ref{A1} and \ref{A2} hold. If, in addition, the derivatives of $b$ and $\sigma$ exist, then there is a constant $C>0$ such that for any $m\leq\min\left\{\frac{c_2}{2L_b}, \frac{2}{c_2}, 1\right\}, $
	\begin{align*}
		\mathbb{W}_1(\nu_m, \nu)\leq C\sqrt{m}.
	\end{align*}
\end{thm}

\section{Auxiliary lemmas}\label{se3}
This section presents several auxiliary lemmas and propositions that will be used in the proofs of the main theorems. 
\subsection{The existence and uniqueness of the invariant distributions}
 Denote $(X_t^{m, x}, Y_t^{m, y})$ by the solution of \eqref{EQ1} with initial condition $(x, y)$ and $X_t^x$   the solution of \eqref{EQ2} with initial condition $x.$  Under Assumptions \ref{A0},  \ref{A1} and \ref{A2}, by a standard argument we can show that the SDEs \eqref{EQ1} and \eqref{EQ2} are both exponentially ergodic.  
\begin{prop}\label{prop2}
Let Assumptions \ref{A0},  \ref{A1} and \ref{A2} hold, and assume $m\leq\min\left\{\frac{c_2}{2L_b}, \frac{2}{c_2}, 1\right\}$. Then the SDE \eqref{EQ1} is exponentially ergodic and admits a  unique invariant distribution $\pi_m.$  More precisely, there exist constants $D_m>0$ and $\rho_m\in(0,1)$  such that for any $(x,y)\in\mathbb{R}^d\times\mathbb{R}^d,$ $t\geq0,$
\begin{align*}	\sup\limits_{|f|\leq\Tilde{V}_m}|\mathbb{E}f(X_t^{m,x},Y_t^{m,y})-\pi_m(f)|\leq D_m\Tilde{V}_m(x,y)\rho_m^t,
\end{align*} 
where $\Tilde{V}_m(x,y)=2|x|^2+6m^2|y|^2+4m\langle x,y\rangle+1.$
\end{prop}

\begin{prop}\label{egeq2}
Let Assumptions \ref{A0}, \ref{A1} and \ref{A2} hold. Then the SDE \eqref{EQ2} is exponentially ergodic and possesses a unique invariant measure $\nu.$ More precisely,  there exist two positive constants $\kappa_1, \kappa_2$ such that 
\begin{align}\label{eg}
	\sup\limits_{|f|\leq 1+|\cdot|^2}[\mathbb{E}f(X_t^x)-\nu(f)]\leq \kappa_1(1+|x|^2)e^{-\kappa_2t},\ \forall x\in\mathbb{R}^d.
\end{align}
This implies
\begin{align}\label{eg1}
	\mathbb{W}_1(\mathcal{L}(X_t^x),\nu)\leq \kappa_1(1+|x|^2)e^{-\kappa_2t}, \ \forall x\in\mathbb{R}^d.
\end{align} 
\end{prop}
The proofs of these propositions rely on standard arguments and are provided in the Section~\ref{seerg} for completeness.
\subsection{Moment estimates}\label{ME1}
In the following we give the moment estimates for the invariant distributions of the equations \eqref{EQ1} and \eqref{EQ2}.
\begin{lem}\label{XMYMesti}
	Let Assumptions \ref{A0},  \ref{A1} and \ref{A2} hold. Assume $(X_0^m, Y^m_0)\sim\pi_m$ and $m\leq\min\left\{\frac{c_2}{2L_b}, \frac{2}{c_2}, 1\right\}$. Then for any $p\ge2,$ there exists constants $C_p$ such that the following estimates hold:

\begin{itemize}
    \item[\rm(1)] $\mathbb{E}|X_t^m|^{p}\leq C_p,\quad t\ge0;$

 \item[\rm(2)] $\mathbb{E}|Y_t^m|^{p}\leq C_p/m^\frac{p}{2}, \quad t\ge0;$

\item[\rm(3)] $\mathbb{E}|X_t^m-X_0^m|^p\leq C_p(t^\frac{p}{2}+m^\frac{p}{2}), \quad t\in[0,1];$
	
\item[\rm(4)] $\mathbb{E}[|X_t^m-X_0^m|^p|\ln|X_t^m-X_0^m||^p]\leq C_p(t^\frac{p}{2}+m^\frac{p}{2})[|\ln(t+m)|^p+1], \quad  t \in [0,1].$
\end{itemize}
\end{lem}

\begin{lem}\label{MOMX}
	Let Assumptions \ref{A0}, \ref{A1} and \ref{A2} hold.  If $X_0\sim\nu,$ then  there exists a constant $C>0$ such that
	\begin{equation}\label{ineq4}
		\mathbb{E}|X_t|^2\leq C, \quad  t\geq0.
	\end{equation} 

\end{lem}

\subsection{Stein's equation }\label{SEE}

The infinitesimal generator associated with the SDE \eqref{EQ1} is
\begin{equation} \label{e:AmGen}
\begin{split}
	\mathcal{A}_mf(x, y)=&\langle y,\nabla_xf(x,y)\rangle+\frac{1}{m}\langle b(x)-y,\nabla_yf(x,y)\rangle \\
    &+\frac{1}{2m^2}\langle \sigma(x)\sigma^T(x),\nabla^2_{yy}f(x,y)\rangle_{\rm HS},  \ \forall f \in C^2_p(\mathbb{R}^d \times \mathbb{R}^d, \mathbb{R}),
\end{split}
\end{equation}
while the infinitesimal generator of the limiting equation \eqref{EQ2} is
\begin{align} \label{e:AGen}
\mathcal{A} g(x) = \langle b(x), \nabla g(x) \rangle + \frac{1}{2} \bigl\langle \sigma(x)\sigma^{\top}(x), \nabla^2 g(x) \bigr\rangle_{\mathrm{HS}}, \ \forall g \in C^2_p(\mathbb{R}^d, \mathbb{R}).
\end{align}

For a given $h\in{\rm Lip}_0(1),$  the corresponding Stein equation is
\begin{equation}\label{stein}
	\mathcal{A}f(x)=h(x)-\nu(h),
\end{equation}
where $\mathcal{A}$ is defined in \eqref{e:AGen}. 
Denote by $\left\{P_t\right\}$  the transition semigroup of \eqref{EQ2}, i.e.,  $P_tf(x):=\mathbb{E}[f(X^x_t)], $ for any $f \in C_p(\mathbb{R}^d, \mathbb{R})$ with $\mathbb{E}f(X_t^x)<\infty.$
\begin{lem}\label{soluofstein}
	For any $h\in{\rm Lip}_0(1),$ the function
	 \begin{equation}\label{solution}
		f(x):=-\int_0^\infty P_t[h(x)-\nu(h)]dt
	\end{equation}
	is a solution to the Stein equation \eqref{stein}.
	Moreover,  there exists a constant $C>0$ such that 
	\begin{equation}\label{linear}
		|f(x)|\leq C(1+|x|^2).
	\end{equation}
	
\end{lem}
\begin{proof}
    Let $\hat{h}=\nu(h)-h$ and we first verify that $f$ is well defined. For any $x\in\mathbb{R}^d$, by the exponential convergence estimate \eqref{eg1} there exist two positive constants $\kappa_1$ and $\kappa_2$ such that 
	\begin{equation*}
		|P_th(x)-\nu(h)|\leq\mathbb{W}_1(P^*_t\delta_x, \nu)\leq \kappa_1 e^{-\kappa_2 t}(1+|x|^2), \  \forall t\geq0
	\end{equation*}
	which implies
    \begin{equation}\label{bound:I}|f(x)|\leq \frac{\kappa_1}{\kappa_2}(1+|x|^2).\end{equation}
     
     For any $\lambda>0$,  the resolvent $\lambda-\mathcal{A}$ of the generator $\mathcal{A}$   
     satisfies
\[
(\lambda - \mathcal{A})^{-1} \hat{h} = \int_0^\infty e^{-\lambda t} P_t \hat{h} \, dt.
\]
Hence, the function \(u_\lambda(x) := \int_0^\infty e^{-\lambda t} P_t \hat{h}(x) \, dt\) belongs to the domain of \(\mathcal{A}\) and
\begin{equation}\label{resolvent}
    \lambda u_\lambda - \hat{h} = \mathcal{A} u_\lambda.
\end{equation}

By the dominated convergence theorem, \(u_\lambda(x) \to f(x)\) as \(\lambda \to 0\) for each \(x\). Moreover, the bound \eqref{bound:I} guarantees that the convergence is dominated by an integrable function (with respect to \(t\)). Since \(\mathcal{A}\) is a closed operator and \(u_\lambda\) converges pointwise to \(f\), passing to the limit \(\lambda \to 0\) in \eqref{resolvent} gives
\[
 \mathcal{A} f(x)=-\hat{h}(x) = h(x)-\nu(h),
\]
which completes the proof.
     
\end{proof}

Next we give the regularity estimates on the solution $f$ of the Stein equation Eq.\eqref{stein}, whose proofs will be presented in Section $5$.
\begin{prop}\label{regularity}
Let Assumptions \ref{A0}, \ref{A1} and \ref{A2} hold, then there exists some positive constant $C$ such that for any $x\in\mathbb{R}^d$
 \begin{equation}
 |\nabla f(x)|\leq C(1+|x|^3)\ \text{and} \  ||\nabla^2f(x)||\leq C(1+|x|^4).
  \end{equation}
 \end{prop}
 \begin{prop}\label{logf}
Let Assumptions \ref{A0},\ref{A1} and \ref{A2} hold, then there exists some positive constant $C$ such that  for any $x\in\mathbb{R}^d$
 \begin{align*}
 \sup\limits_{\left\{y:0<|y-x|\leq\frac{1}{8}\right\}}\frac{||\nabla^2f(x)-\nabla^2f(y)||}{|x-y||\ln|x-y||}\leq C(1+|x|^5).
 \end{align*}
\end{prop}


	
	  

\section{Proof of Main results}\label{se4}

We organize the proofs of our main results, Theorems \ref{thm1} and \ref{thm2}, into the following two subsections. Throughout this section, we assume that the stochastic process $(X^m_t, Y^m_t)_{t \ge 0}$ is stationary. In the derivations below, $C$ denotes a generic positive constant that may change from line to line but is independent of $t$ and $m$.

\begin{proof}
Let $(X_0^m, Y_0^m) \sim \pi_m$ with $\pi_m$ being the invariant measure of \eqref{EQ1}, then $\{(X_t^m, Y_t^m)\}_{t \ge 0}$ is a stationary Markov process and $(X_t^m, Y_t^m) \sim \pi_m$ for all $t \ge 0$. 
For any given $h \in {\rm Lip_0(1)}$, let $f$ be the solution to Stein's equation \eqref{stein} and define
$$g(x,y)=\langle y,\nabla f(x)\rangle.$$
Denote by $\nu_m(dx)=\pi_m(dx,\mathbb{R}^d)$  the marginal distribution of $\pi_m$ with respect to the position variable $x$.
\vskip 2.5mm 
\underline{\emph{Step 1}}: To show that 
\begin{equation} \label{phixy}
\nu_m(h)-\nu(h)=\pi_m(\phi)
\end{equation}
where 
\begin{equation}
\phi(x, y)=\frac{1}{2}\langle\sigma(x)\sigma^T(x),\nabla^2f(x)\rangle_{\rm HS}-my^T\nabla^2f(x)y. 
\end{equation}
Here
$$\nu_m(h)=\int_{\mathbb{R}^{d}} h(x) \nu_m(dx), \ \ \ \ \pi_m(\phi)=\int_{\mathbb{R}^{2d}} \phi(x,y) \pi_m(dx,dy).$$

By applying It\^o's formula, we have for all $t\geq0$
\begin{equation} \label{e:ItoG}
\begin{split}
	0=&\mathbb{E}[g(X_t^m, Y_t^m)]-\mathbb{E}[g(X_0^m,Y_0^m)]\\
    =&\mathbb{E}[\int_0^t\mathcal{A}_m g(X_s^m,Y_s^m)ds]\\
	=&\int_0^t\mathbb{E}[(Y_s^m)^T\nabla^2f(X_s^m)Y_s^m+\langle\frac{1}{m}(b(X_s^m)-Y_s^m), \nabla f(X_s^m)\rangle] ds\\
	=&\int_0^t\mathbb{E}[(Y_s^m)^T\nabla^2f(X_s^m)Y_s^m+\frac{1}{m}\langle b(X_s^m), \nabla f(X_s^m)\rangle] ds.
\end{split}
\end{equation}
Here the last equality follows from the stationary identity: 
\begin{align}\label{ypartialx}
	\mathbb{E}\left[\mathcal{A}_m f(X^m_s,Y^m_s)\right]=\mathbb{E}\langle Y_s^m,\nabla f(X_s^m)\rangle=0,
\end{align} 
which is a consequence of the invariance of $\pi_m$ under the generator $\mathcal{A}_m$.

Therefore we get for all $t\geq0,$
\begin{equation} \label{e:BYGradf}
\mathbb{E}[\langle b(X_t^m), \nabla f(X_t^m)\rangle]=-m \mathbb{E}[(Y_t^m)^T\nabla^2f(X_t^m)Y_t^m\rangle].
\end{equation}

Now consider the Stein equation \eqref{stein} at $x=X^m_t$. By taking expectation and using \eqref{e:BYGradf}, we obtain that for all $t\geq0$
\begin{align*}
	\nu_m(h)-\nu(h)&=\mathbb{E}[h(X^m_t)]-\nu(h) \nonumber \\
    &=\mathbb{E}\left[\frac{1}{2}\langle\sigma(X^m_t)\sigma^T(X^m_t), \nabla^2f(X^m_t)\rangle_{\rm HS}+\langle b(X^m_t), \nabla f(X^m_t) \rangle\right] \nonumber \\
    &=\mathbb{E}\left[\frac{1}{2}\langle\sigma(X^m_t)\sigma^T(X^m_t), \nabla^2f(X^m_t)\rangle_{\rm HS}-m (Y_t^m)^T\nabla^2f(X_t^m)Y_t^m\right], \nonumber
\end{align*}
which implies \eqref{phixy}. 
\vskip 2mm

\underline{\emph{Step 2}}: To show that for all $t>0$, 
\begin{align} \label{e:PhiXtYt}
\pi_m(\phi)=\frac{m}{2t}[I_1+I_2+I_3+I_4+I_5+I_6],
\end{align}
where 
\begin{align*}
& I_1=-\mathbb{E}\langle\nabla^2f(X_t^m), \int_0^tY_s^mb^T(X_s^m)+b(X_s^m)(Y_s^m)^Tds\rangle_{\rm HS},\\
&I_2=-\mathbb{E}\langle\nabla^2f(X_t^m), \int_0^tY_s^m(\sigma(X_s)dB_s)^T+(\sigma(X_s)dB_s)(Y_s^m)^T\rangle_{\rm HS},\\
&I_3=2\mathbb{E}\int_0^t\langle\nabla^2f(X^m_t)-\nabla^2f(X^m_s), Y_s^m(Y_s^m)^T\rangle_{\rm HS}ds, \\
&I_4=-\frac{1}{m}\mathbb{E}\int_0^t\langle\nabla^2f(X_t^m)-\nabla^2f(X_s^m), \sigma(X_s^m)\sigma^T(X_s^m)\rangle_{\rm HS}ds,\\
&I_5=\frac{1}{2}\mathbb{E}\langle\sigma(X_t^m)\sigma^T(X_t^m)-\sigma(X_0^m)\sigma^T(X_0^m), \nabla^2f(X_t^m)\rangle_{\rm HS},
\\
&I_6=\mathbb{E}\langle\frac{1}{2}\sigma(X_0^m)\sigma^T(X_0^m)-mY_0^m(Y_0^m)^T, \nabla^2f(X_t^m)-\nabla^2f(X_0^m)\rangle_{\rm HS}.
\end{align*}
Using the stationarity of the process $\{(X^m_t,Y^m_t)\}_{t \ge 0}$, we have for any $t\geq0$
\begin{align*}
	0=\mathbb{E}[\phi(X_t^m, Y_t^m)-\phi(X_0^m, Y_0^m)]=J_1+J_2, 
\end{align*}
where $$J_1=\mathbb{E}[\phi(X_t^m, Y_t^m)-\phi(X_t^m, Y_0^m)],\ \  \ \ J_2=\mathbb{E}[\phi(X_t^m, Y_0^m)-\phi(X_0^m, Y_0^m)].$$
For $J_1$, it is easy to see that   
\begin{align*}
	J_1=&-m\mathbb{E}[(Y_t^m)^T\nabla^2f(X_t^m)Y_t^m-(Y_0^m)^T\nabla^2f(X_t^m)Y_0^m]\\
	=&-m\mathbb{E}[\langle\nabla^2f(X_t^m), Y_t^m(Y_t^m)^T-Y_0^m(Y_0^m)^T\rangle_{\rm HS}].
\end{align*}
By It\^o's formula, 
\begin{align*}
	Y_t^m(Y_t^m)^T-Y_0^m(Y_0^m)^T&=\frac 1m \int_0^t \left[Y_s^mb^T(X_s^m)+b(X_s^m)(Y_s^m)^T-2 Y^m_s (Y^m_s)^T+\frac 1m \sigma(X_s^m)\sigma^T(X_s^m)\right]ds\\
	&\ \ +\int_0^tY_t^m\frac{1}{m}(\sigma(X_s^m)dB_s)^T+\int_0^t\frac{1}{m}(\sigma(X_s^m)dB_s)(Y_s^m)^T.
\end{align*}
Consequently, 
\begin{align*}
	\phi(X_t^m, Y_t^m)-\phi(X_t^m, Y_0^m)&=-\langle\nabla^2f(X_t^m), \int_0^tY_s^mb^T(X_s^m)+b(X_s^m)(Y_s^m)^Tds\rangle_{\rm HS}\\
	&+2\langle\nabla^2f(X_t^m), \int_0^tY_s^m(Y_s^m)^Tds\rangle_{\rm HS}-\frac{1}{m}\langle\nabla^2f(X_t^m), \int_0^t\sigma(X_s^m)\sigma^T(X_s^m)ds\rangle_{\rm HS} \\
    &-\langle\nabla^2f(X_t^m), \int_0^tY_s^m(\sigma(X^m_s)dB_s)^T+(\sigma(X^m_s)dB_s)(Y_s^m)^T\rangle_{\rm HS}.
\end{align*}
We now rewrite the second and third terms as follows:
\begin{align*}
	&2\langle\nabla^2f(X_t^m), \int_0^tY_s^m(Y_s^m)^Tds\rangle_{\rm HS}-\frac{1}{m}\langle\nabla^2f(X_t^m), \int_0^t\sigma(X_s^m)\sigma^T(X_s^m)ds\rangle_{\rm HS}\\
	=&2\int_0^t\langle\nabla^2f(X^m_t)-\nabla^2f(X^m_s), Y_s^m(Y_s^m)^T\rangle_{\rm HS} ds+2\int_0^t\langle\nabla^2f(X^m_s), Y_s^m(Y_s^m)^T\rangle_{\rm HS} ds\\
	-&\frac{1}{m}\int_0^t\langle\nabla^2f(X_t^m)-\nabla^2f(X_s^m), \sigma(X_s^m)\sigma^T(X_s^m)\rangle_{\rm HS} ds-\frac{1}{m}\int_0^t\langle\nabla^2f(X_s^m), \sigma(X_s^m)\sigma^T(X_s^m)\rangle_{\rm HS} ds\\
	=&2\int_0^t\langle\nabla^2f(X^m_t)-\nabla^2f(X^m_s), Y_s^m(Y_s^m)^T\rangle_{\rm HS}ds-\frac{1}{m}\int_0^t\langle\nabla^2f(X_t^m)-\nabla^2f(X_s^m), \sigma(X_s^m)\sigma^T(X_s^m)\rangle_{\rm HS} ds\\
	-&\frac{2}{m}\int_0^t\phi(X_s^m, Y_s^m)ds, 
\end{align*}
where the last equality follows from the definition of $\phi$. 

Thus $J_1$  can be expressed as
\begin{align*}
	J_1=&\mathbb{E}[-\langle\nabla^2f(X_t^m), \int_0^tY_s^mb^T(X_s^m)+b(X_s^m)(Y_s^m)^Tds\rangle_{\rm HS}\\
	&-\langle\nabla^2f(X_t^m), \int_0^tY_s^m(\sigma(X_s)dB_s)^T+(\sigma(X_s)dB_s)(Y_s^m)^T\rangle_{\rm HS}\\
	&+2\int_0^t\langle\nabla^2f(X^m_t)-\nabla^2f(X^m_s), Y_s^m(Y_s^m)^T\rangle_{\rm HS}ds-\frac{1}{m}\int_0^t\langle\nabla^2f(X_t^m)-\nabla^2f(X_s^m), \sigma(X_s^m)\sigma^T(X_s^m)\rangle_{\rm HS}ds\\
	&-\frac{2}{m}\int_0^t\phi(X_s^m, Y_s^m)ds].
\end{align*}
For $J_2$, we have directly  
\begin{align*}
	J_2=&\frac{1}{2}\langle\sigma(X_t^m)\sigma^T(X_t^m), \nabla^2f(X_t^m)\rangle_{\rm HS}-\frac{1}{2}\langle\sigma(X_0^m)\sigma^T(X_0^m), \nabla^2f(X_0^m)\rangle_{\rm HS}\\
	&-m\langle\nabla^2f(X_t^m)-\nabla^2f(X_0^m), Y_0^m(Y_0^m)^T\rangle_{\rm HS}\\
	=&\frac{1}{2}\langle\sigma(X_t^m)\sigma^T(X_t^m)-\sigma(X_0^m)\sigma^T(X_0^m), \nabla^2f(X_t^m)\rangle_{\rm HS}
	\\&+\langle\frac{1}{2}\sigma(X_0^m)\sigma^T(X_0^m)-mY_0^m(Y_0^m)^T, \nabla^2f(X_t^m)-\nabla^2f(X_0^m)\rangle_{\rm HS}.
\end{align*}
Combining the expressions for $J_1$ and $J_2$ and using $J_1+J_2=0$, we obtain
\begin{align*}
&\mathbb{E}\frac{2}{m}\int_0^t\phi(X_s^m, Y_s^m)ds=\mathbb{E}[-\langle\nabla^2f(X_t^m), \int_0^tY_s^mb^T(X_s^m)+b(X_s^m)(Y_s^m)^Tds\rangle_{\rm HS}\\
-&\langle\nabla^2f(X_t^m), \int_0^tY_s^m(\sigma(X_s)dB_s)^T+(\sigma(X_s)dB_s)(Y_s^m)^T\rangle_{\rm HS}\\
+&2\int_0^t\langle\nabla^2f(X^m_t)-\nabla^2f(X^m_s), Y_s^m(Y_s^m)^T\rangle_{\rm HS}ds-\frac{1}{m}\int_0^t\langle\nabla^2f(X_t^m)-\nabla^2f(X_s^m), \sigma(X_s^m)\sigma^T(X_s^m)\rangle_{\rm HS}ds\\
+&\frac{1}{2}\langle\sigma(X_t^m)\sigma^T(X_t^m)-\sigma(X_0^m)\sigma^T(X_0^m), \nabla^2f(X_t^m)\rangle_{\rm HS}
\\+&\langle\frac{1}{2}\sigma(X_0^m)\sigma^T(X_0^m)-mY_0^m(Y_0^m)^T, \nabla^2f(X_t^m)-\nabla^2f(X_0^m)\rangle_{\rm HS}].
\end{align*}
By the stationary property of $\{(X^m_t,Y^m_t)\}_{t \ge 0}$, we have 
$$\mathbb{E}\frac{2}{m}\int_0^t\phi(X_s^m, Y_s^m)ds=\frac{2 t}{m} \mathbb{E} \phi(X_t^m, Y_t^m)=\frac{2 t}{m} \pi_m(\phi).$$
By the two equalities above, we finally get \eqref{e:PhiXtYt}. 

\vskip 2mm 
\underline{\emph{Step 3}}: To estimate $I_1$-$I_6$.
For $I_1$,  recall the following regularity estimate in Proposition~\ref{regularity}: $$||\nabla^2f(x)||\leq C(1+|x|^4),$$ 
we have
\begin{align*}
	|I_1| & \leq C \int_0^t\mathbb{E}(1+|X_t^m|^4)|b(X_s^m)||Y_s^m|ds \\
    & \leq C\int_0^t \sqrt{\mathbb{E}[(1+|X_t^m|^4)^2 |b(X_s^m)|^2]} \sqrt{\mathbb{E}|Y_s^m|^2} ds,
\end{align*}
where the second inequality follows from the Cauchy-Schwarz inequality.

By the linear growth property of $b$ and Lemma~\ref{XMYMesti}~(1), the first term in the integrand is bounded. By Lemma~\ref{XMYMesti}~(2), there exists some positive constant $C$ such that $\mathbb{E}|Y_s^m|^2\le \frac{C}{m}$. Thus we get \begin{align*}
	|I_1| \leq C\frac{t}{\sqrt{m}}.
\end{align*}

For $I_2$, we decompose it as 
$$I_2=I_{21}+I_{22}$$
where
\begin{align*}I_{21}&=-\mathbb{E}\langle\nabla^2f(X_t^m), \int_0^tY_s^m(\sigma(X^m_s)dB_s)^T\rangle_{\rm HS}, \\
I_{22}&=-\mathbb{E}\langle\nabla^2f(X_t^m), \int_0^t (\sigma(X^m_s)dB_s)(Y_s^m)^T\rangle_{\rm HS}.\end{align*}
Since 
	$\mathbb{E}\langle\nabla^2f(X_0^m), \int_0^tY_s^m(\sigma(X^m_s)dB_s)^T\rangle_{\rm HS}=0, $  we have
\begin{align*}
	|I_{21}|=&|\mathbb{E}\langle\nabla^2f(X_t^m), \int_0^tY_s^m(\sigma(X^m_s)dB_s)^T\rangle_{\rm HS}|\\
	=&|\mathbb{E}\langle\nabla^2f(X_t^m)-\nabla^2f(X_0^m), \int_0^tY_s^m(\sigma(X^m_s)dB_s)^T\rangle_{\rm HS}|\\
	\leq&\mathbb{E}||\nabla^2f(X_t^m)-\nabla^2f(X_0^m)|||\int_0^t\langle Y_s^m, \sigma(X^m_s)dB_s\rangle|\\
    \leq& C\sqrt{\mathbb{E}||\nabla^2f(X_t^m)-\nabla^2f(X_0^m)||^2} \sqrt{\int_0^t \mathbb{E}|\sigma(X^m_s)Y_s^m|^2 ds}
\end{align*}
where the last inequality is obtained by the Cauchy-Schwarz inequality and It\^{o}'s isometry. By Assumption \ref{A0} and Lemma \ref{XMYMesti}~(2), we have
\begin{align} \label{e:I21-First}
    {\int_0^t \mathbb{E}|\sigma(X^m_s)Y_s^m|^2 ds}\leq C {\int_0^t \mathbb{E}|Y_s^m|^2 ds} \le C\frac{t}{m}.
\end{align}

Next we estimate $\mathbb{E}||\nabla^2f(X_t^m)-\nabla^2f(X_0^m)||^2$. Split the expectation according to the size of $|X_t^m-X_0^m|$:
\begin{align*}
	\mathbb{E}||\nabla^2f(X_t^m)-\nabla^2f(X_0^m)||^2
    =&\mathbb{E}||\nabla^2f(X_t^m)-\nabla^2f(X_0^m)||^2\mathbb{I}_{\left\{|X_t^m-X_0^m|\geq\frac{1}{8}\right\}}\\
    &+\mathbb{E}||\nabla^2f(X_t^m)-\nabla^2f(X_0^m)||^2\mathbb{I}_{\left\{0<|X_t^m-X_0^m|<\frac{1}{8}\right\}}.
\end{align*}
For the first part, using the regularity estimate in Proposition~\ref{regularity} and the Cauchy-Schwarz inequality, we have
\begin{align*}
    &\mathbb{E}\|\nabla^2f(X_t^m)-\nabla^2f(X_0^m)\|^2\mathbb{I}_{\left\{|X_t^m-X_0^m|\geq\frac{1}{8}\right\}} \\
    \le &C\mathbb{E}\left[\left(\|\nabla^2f(X_t^m)\|^2+\|\nabla^2f(X_0^m)\|^2\right)\mathbb{I}_{\left\{|X_t^m-X_0^m|\geq\frac{1}{8}\right\}}\right] \\
    \leq & C\mathbb{E}[(1+|X_t^m|^4+|X_0^m|^4)\mathbb{I}_{\left\{|X_t^m-X_0^m|\geq\frac{1}{8}\right\}}]\\
    \leq & C\sqrt{\mathbb{E}(1+|X_t^m|^4+|X_0^m|^4)^2} \sqrt{\mathbb{P}\left\{|X_t^m-X_0^m|\geq1/8\right\}}.
\end{align*}
By using Lemma \ref{XMYMesti}~(1), (3) and Markov's inequality, we obtain
\begin{equation} \label{e:E>1/8}
\mathbb{E}\|\nabla^2f(X_t^m)-\nabla^2f(X_0^m)\|^2\mathbb{I}_{\left\{|X_t^m-X_0^m|\geq\frac{1}{8}\right\}}     \leq  C\sqrt{\mathbb{E}|X_t^m-X_0^m|^2} 
    \leq  C(t+m).
\end{equation}
For the second part, by Proposition \ref{logf}, we have
\begin{equation}
\begin{split}
   & \mathbb{E}\|\nabla^2f(X_t^m)-\nabla^2f(X_0^m)\|^2\mathbb{I}_{\left\{0<|X_t^m-X_0^m|<\frac{1}{8}\right\}} \\
    \leq &\mathbb{E} \left[\frac{\|\nabla^2f(X_t^m)-\nabla^2f(X_0^m)\|^2}{|X_t^m-X^m_0| |\ln |X_t^m-X^m_0||}|X_t^m-X^m_0| |\ln |X_t^m-X^m_0||\mathbb{I}_{\left\{0<|X_t^m-X_0^m|<\frac{1}{8}\right\}}\right] \\
   \le & C \mathbb{E}\left[(1+|X^m_t|^5) |X_t^m-X^m_0| |\ln |X_t^m-X^m_0||\right]. \\
\end{split}
\end{equation}
Applying H\"{o}lder’s inequality together with Lemmas~\ref{XMYMesti}~(1) and (4), we get
\begin{equation} \label{e:E<1/8}
    \mathbb{E}||\nabla^2f(X_t^m)-\nabla^2f(X_0^m)||^2\mathbb{I}_{\left\{0<|X_t^m-X_0^m|<\frac{1}{8}\right\}}\leq C(t+m)\left[1+|\ln(t+m)|^2\right].
\end{equation}
Combining \eqref{e:E>1/8} and \eqref{e:E<1/8}, we obtain
\begin{align} \label{e:ELab^2FDif}
    \mathbb{E}\|\nabla^2f(X_t^m)-\nabla^2f(X_0^m)\|^2\leq C(t+m)
    |\ln(t+m)|^2+C(t+m).
\end{align}
Finally, substituting \eqref{e:I21-First} and \eqref{e:ELab^2FDif} into the bound for $I_{21}$, we get
\begin{align*}
    |I_{21}|\leq C\frac{\sqrt{t}}{\sqrt{m}}[(\sqrt{t}+\sqrt{m})|\ln(t+m)|+\sqrt{t}+\sqrt{m}].
\end{align*}
The term $I_{22}$ can be estimated in exactly the same way and admits the same bound as $I_{21}$. Hence 
\begin{align*}
	|I_2|\leq C\frac{\sqrt{t}}{\sqrt{m}}[(\sqrt{t}+\sqrt{m})|\ln(t+m)|+\sqrt{t}+\sqrt{m}].
\end{align*}

For $I_3$ and $I_4$, by \eqref{e:ELab^2FDif} and Lemmas \ref{XMYMesti}~(2),  we have 
\begin{align*}
		|I_3|\leq&C\int_0^t\sqrt{\mathbb{E}||\nabla^2f(X_t^m)-\nabla^2f(X_s^m)||^2\mathbb{E}|Y_s^m|^4}ds\\
        \leq&C\frac{t}{m}[(\sqrt{t}+\sqrt{m})|\ln(t+m)|+\sqrt{t}+\sqrt{m}],
\end{align*}
while $I_4$ can be bounded by the same argument, yielding
\begin{align*}
&|I_4|\leq C\frac{t}{m}[(\sqrt{t}+\sqrt{m})|\ln(t+m)|+\sqrt{t}+\sqrt{m}].
\end{align*}

For $I_5$ and $I_6$, by Assumption \ref{A0} and Proposition \ref{regularity} and Lemma \ref{XMYMesti}~(3),  
\begin{align*}
		|I_5|\leq C\mathbb{E}[(1+|X_t^m|^4)|X_t^m-X_0^m|]\leq C\sqrt{\mathbb{E}|X_t^m-X_0^m|^2}\leq C(\sqrt{t}+\sqrt{m}),
\end{align*}
moreover,
by Assumption \ref{A0}, Lemmas \ref{XMYMesti}~(2)(3) and \eqref{e:ELab^2FDif}, we get
\begin{align*}
		|I_6|\leq& C\sqrt{\mathbb{E}\|\sigma(X_0^m)\sigma^T(X_0^m)\|^2_{\rm HS}+m^2\mathbb{E}\|Y_0^m(Y_0^m)^T\|_{\rm HS}^2}\sqrt{\mathbb{E}\|\nabla^2f(X_t^m)-\nabla^2f(X_0^m)\|_{\rm HS}^2}\\
        \leq&C\sqrt{\mathbb{E}\|\nabla^2f(X_t^m)-\nabla^2f(X_0^m)\|_{\rm HS}^2}\leq C[(\sqrt{t}+\sqrt{m})|\ln(t+m)|+\sqrt{t}+\sqrt{m}]
\end{align*}
\vskip 2.5mm
\underline{\emph{Step 4}}: Completion of the proof.  
Combining the estimates for $I_1$-$I_6$, we have \begin{align*}
	&|\nu_m(h)-\nu(h)|=|\pi_m(\phi)|\\
    \le& \frac{m}{2t}[|I_1|+|I_2|+|I_3|+|I_4|+|I_5|+|I_6|] \\
     \le &C[\sqrt{m}+(\frac{\sqrt{m}}{\sqrt{t}}+1+\frac{m}{t})[(\sqrt{t}+\sqrt{m})|\ln(t+m)|+\sqrt{t}+\sqrt{m}]+\frac{m}{t}(\sqrt{t}+\sqrt{m})]
\end{align*}
Taking $t=m$, we get
\begin{align*}
	|\nu_m(h)-\nu(h)|
	\leq C[\sqrt{m}+\sqrt{m}|\ln m|].
\end{align*} 
As $m\leq e^{-1}, $ it yields
\begin{align*}
	|\nu_m(h)-\nu(h)|\leq C\sqrt{m}|\ln m|,
\end{align*}
which completes the proof by Kantorovich duality.
\end{proof}

\subsection{Proof of Theorem \ref{thm2}  (One-dimensional case)}
\begin{proof}
Let $h\in {\rm Lip}_0(1)\cap C^1(\mathbb{R})$, then $||h'||_\infty\leq1.$	
Consider the Stein equation $$\mathcal{A}f(x)=h(x)-\nu(h), \ \forall x\in\mathbb{R},$$
where $\mathcal{A}f(x)=b(x)f'(x)+\frac12 \sigma^2(x)f''(x)$.

By Lemma \ref{soluofstein}, 
$$f(x)=-\int_0^\infty P_t[h(x)-\nu(h)]dt.$$
Differentiating the Stein equation $\eqref{stein}$, one  yields:
\begin{align*}
	b'(x)f'(x)+b(x)f''(x)+\sigma(x)\sigma'(x)f''(x)+\frac{1}{2}\sigma^2(x)f'''(x)=h'(x),
\end{align*}
and 
\begin{align*}
	f'''(x)=\frac{2[h'(x)-b'(x)f'(x)-b(x)f''(x)-\sigma(x)\sigma'(x)f''(x)]}{\sigma^2(x)}
\end{align*}
By Assumptions \ref{A0}, \ref{A1},  Proposition \ref{regularity},  we obtain the key estimate:
\begin{align}\label{third}
	|f'''(x)|\leq C(1+|x|^5).
\end{align}

Now let $(X_0^m, Y_0^m) \sim \pi_m$ with $\pi_m$ being the invariant measure of \eqref{EQ1}, then $\{(X_t^m, Y_t^m)\}_{t \ge 0}$ is a stationary Markov process and $(X_t^m, Y_t^m) \sim \pi_m$ for all $t \ge 0$.
From identity \eqref{phixy}, we have
	\begin{align*}
		\nu_m(h)-\nu(h)=\mathbb{E}[\phi(X_t^m,Y_t^m)],
	\end{align*}
where $\phi(x, y)=\frac{1}{2}\sigma^2(x)f''(x)-m y^2f''(x).$ 

The infinitesimal generator $\mathcal{A}_m\phi$ of the two-dimensional kinetic process acts on $\phi$ as
	\begin{align*}
		\mathcal{A}_m\phi(x, y)=&y\partial_x\phi(x, y)-\frac{1}{m}m(b(x)-y)(2f''(x)y)-\frac{1}{2m}(2\sigma^2(x)f''(x))\\
		=&y\partial_x\phi(x, y)-2b(x)f''(x)y+2y^2f''(x)-\frac{1}{2m}(2\sigma^2(x)f''(x))\\
		=&y\partial_x\phi(x, y)-2b(x)f''(x)y+\frac{2}{m}(my^2f''(x)-\frac{1}{2}\sigma^2(x)f''(x))\\
		=&y\partial_x\phi(x, y)-2b(x)f''(x)y-\frac{2}{m}\phi(x, y).
	\end{align*}
	By the stationary property of $(X^m_t,Y^m_t)$, we have 
    $\mathbb{E}\mathcal{A}_m\phi(X^m_t, Y^m_t)=0$ for all $t \ge 0$ and thus
    $$\mathbb{E}\phi(X_t^m, Y_t^m)
	=\mathbb{E}[\frac{m}{2}Y_t^m\partial_x\phi(X_t^m, Y_t^m)-mb(X_t^m)f''(X_t^m)Y_t^m].
	$$
Hence, 
\begin{align*}
	\nu_m(h)-\nu(h)&=\mathbb{E}[\frac{m}{2}Y_t^m\partial_x\phi(X_t^m, Y_t^m)-mb(X_t^m)f''(X_t^m)Y_t^m]\\
	&=\mathbb{E}[Y_t^m(\sigma^2(X_t^m)f''(X_t^m))'-\frac{m^2}{2}(Y_t^m)^3f'''(X_t^m)-mb(X_t^m)f''(X_t^m)Y_t^m].
\end{align*}
Noting that $$\mathbb{E}[Y_t^m(\sigma^2(X_t^m)f''(X_t^m))']=\mathbb{E}\left[\mathcal{A}_m \left(\sigma^2(X_t^m)f''(X_t^m))\right)\right]=0,$$  
we obtain $$\nu_m(h)-\nu(h)=-\mathbb{E}[\frac{m^2}{2}(Y_t^m)^3f'''(X_t^m)+mb(X_t^m)f''(X_t^m)Y_t^m].$$  

Therefore,  by Proposition \ref{regularity},  Lemma \ref{XMYMesti}~(1) and (2),   Assumption \ref{A0},  \eqref{third} and H\"older's inequality,  we have
\begin{align*}
	|\nu_m(h)-\nu(h)|\leq& \frac{m^2}{2}\mathbb{E}[|Y_t^m|^3|f'''(X_t^m)|]+m\mathbb{E}[|b(X_t^m)||f''(X_t^m)||Y_t^m|]\\
	\leq& Cm^2\mathbb{E}[|Y_t^m|^3(1+|X_t^m|^5)]+m\mathbb{E}[|b(X_t^m)||Y_t^m|(1+|X_t^m|^4)]\\
    \leq&Cm^2\mathbb{E}[|Y_t^m|^3(1+|X_t^m|^5)]+m\mathbb{E}[|Y_t^m|(1+|X_t^m|^5)]\\
	\leq&C[m^2(\mathbb{E}|Y_t^m|^4)^\frac{4}{3}(\mathbb{E}(1+|X_t|)^5)^\frac{1}{4}+m\sqrt{\mathbb{E}(1+|X_t^m|^5)^2}\sqrt{\mathbb{E}|Y_t^m|^2}]\\
	\leq& C[m^2(\mathbb{E}|Y_t^m|^4)^\frac{4}{3}+m\sqrt{\mathbb{E}|Y_t^m|^2}]\\
	\leq& C[m^2m^{-\frac{3}{2}}+mm^{-\frac{1}{2}}]\\
	\leq&C\sqrt{m}.
\end{align*}
A standard density argument extends this estimate to every $h\in{\rm Lip}_0(1).$

Taking the supremum over all such $h$ and using the Kantorovich duality for the $1$-Wasserstein distance, we obtain the desired result.
\end{proof}
\section{Proof of ergodicity of Eq.\eqref{EQ1} and Eq.\eqref{EQ2}}\label{seerg}

\subsection{Proof of Proposition \ref{prop2}}

\begin{proof}[Proof of Proposition \ref{prop2}]

 We establish the exponential ergodicity of system \eqref{EQ1} by verifying a Lyapunov condition and using well-known results from the theory of Markov processes.

Consider the Lyapunov function $$V_m(x,y)=2|x|^2+6m^2|y|^2+4m\langle x,y\rangle.$$
A direct calculation shows that 
	\begin{equation}\label{Lyapunov}
		0\leq\frac{1}{8}V_m(x,y)\leq |x|^2+m^2|y|^2\leq V_m(x,y).
	\end{equation}
Applying the generator \(\mathcal{A}_m\) to \(V_m\) and using Assumptions \ref{A0}--\ref{A2}, we obtain
	\begin{align*}
		\mathcal{A}_mV_m(x, y)=&-8m|y|^2+12m\langle b(x), y\rangle+4\langle b(x), x\rangle+6\langle\sigma(x), \sigma(x)\rangle_{\rm HS}\\
		\leq&(6mL_b-4c_2)|x|^2-\frac{2}{m}m^2|y|^2+4c_1+6mL_b+6||\sigma||^2_\infty.
	\end{align*}
    For \(m\le\min\bigl\{\frac{c_2}{2L_b},\frac{2}{c_2},1\bigr\}\) we have \(6mL_b\le 3c_2\) and \(1/m\ge c_2/2\); hence
\[
\mathcal{A}_m V_m(x,y)
\le -c_2\bigl(|x|^2+m^2|y|^2\bigr)+4c_1+3c_2+6\|\sigma\|_\infty^2.
\]
Using \eqref{Lyapunov} this becomes
\begin{equation}\label{AMVM}
    \mathcal{A}_m V_m(x,y)\le -\frac{c_2}{8}\,V_m(x,y)+C,
\end{equation}
where \(C:=4c_1+3c_2+6\|\sigma\|_\infty^2\). The inequality \eqref{AMVM} is a standard dissipative Lyapunov condition, which, together with Kolmogorov backward equation, implies 
$$\frac{d}{dt} \mathbb{E}V_m(X_t^m,Y_t^m)= \mathcal{A}_m \mathbb{E}V_m(X_t^m,Y_t^m) \le -\frac{c_2}{8}\,\mathbb{E}V_m(X_t^m,Y_t^m)+C.$$
Hence,  
\begin{equation}\label{XYbounded}
    \mathbb{E}V_m(X_t^m,Y_t^m)\le e^{-c_2 t/8}\,\mathbb{E}V_m(X_0^m,Y_0^m)+C.
\end{equation}
Consequently, the solution \((X_t^m,Y_t^m)\) exists for all time and its moments are uniformly bounded in time when started from a distribution with finite moments.

To obtain exponential convergence to a unique invariant measure, we need to verify that the semigroup generated by \eqref{EQ1} is strongly Feller and that the process is irreducible.  Since the Harnack inequality implies both the strong Feller property and the irreducibility \cite{Wang2013}, we only need to show that the SDE \eqref{EQ1} satisfies this inequality.  Although the diffusion coefficient \(\sigma\) depends on the position variable \(x\), the uniform ellipticity and boundedness of \(\sigma\) guarantee that the Harnack inequality proved in \cite[Theorem 1.1]{WaZh13} remains valid for system \eqref{EQ1} after minor modifications. The dependence on \(x\) requires only technical adjustments in the Malliavin calculus and control arguments used in \cite{WaZh13}; the core estimates are unchanged. In particular, Assumption (H) of \cite{WaZh13} holds with respect to the Lyapunov function \(V_m\). The regularity assumptions on the coefficients therein can be removed by a standard approximation and limiting procedure.

Now to show the ergodicity of SDE \eqref{EQ1} by the results in \cite[Theorem 2.4]{Wu2001}, it remains to prove that there exists a Lyapunov function $V\geq1,$  a compact $K$ and positive constants $\varepsilon,C$ such that 
\begin{align*}
    -\frac{\mathcal{A}_mV}{V}\geq \varepsilon\mathbb{I}_{K^C}-C\mathbb{I}_{K}.
\end{align*}


Let $\Tilde{V}_m(x,y):=V_m(x,y)+1\geq1.$ By \eqref{AMVM}, we have 
\begin{align*}
    \mathcal{A}_m\Tilde{V}_m(x,y)\leq&-c_2(|x|^2+m^2|y|^2+1)+4c_1+c_2+6L_b+6||\sigma||^2_\infty\\
    \leq&[-c_2(|x|^2+m^2|y|^2+1)+c']\mathbb{I}_{K_m}(x,y)+[-c_2(|x|^2+m^2|y|^2+1)+c']\mathbb{I}_{K_m^C}(x,y).
\end{align*}
where $c':=4c_1+c_2+6L_b+6||\sigma||^2_\infty$, and  $K_m:=\left\{(x,y)||x|^2+m^2|y|^2+1\leq \frac{2c'}{c_2}\right\}$ is a compact set. 

On the one hand, it is easy to see that
\begin{align*}
    [-c_2(|x|^2+m^2|y|^2+1)+c']\mathbb{I}_{K_m}\leq c'\mathbb{I}_{K_m}\leq c'\Tilde{V}_m(x,y)\mathbb{I}_{K_m}(x,y).
\end{align*}
On the other hand, by the definition of $K$ and \eqref{Lyapunov}, we have
\begin{align*}
[-c_2(|x|^2+m^2|y|^2+1)+c']\mathbb{I}_{K_m^C}(x,y)&\leq-\frac{c_2}{2}(|x|^2+m^2|y|^2+1)\mathbb{I}_{K_m^C}(x,y)\\
&\leq-\frac{c_2}{16}\Tilde{V}_m(x,y)\mathbb{I}_{K_m^C}(x,y).
\end{align*}
Combing the above estimates together, we obtain
\begin{align*}
    \mathcal{A}_m\Tilde{V}_m(x,y)\leq c'\Tilde{V}_m(x,y)\mathbb{I}_{K_m}(x,y)-\frac{c_2}{16}\Tilde{V}_m(x,y)\mathbb{I}_{K_m^C}(x,y),
\end{align*}
which implies $$-\frac{\mathcal{A}_m\Tilde{V}_m(x,y)}{\Tilde{V}_m(x,y)}\geq\frac{c_2}{16}\Tilde{V}_m(x,y)\mathbb{I}_{K_m^C}(x,y)-c'\Tilde{V}_m(x,y)\mathbb{I}_{K_m}(x,y).$$

Then, by \cite[Theorem 2.4]{Wu2001}, there exists a unique invariant measure $\pi_m$,  constants $D_m>0$ and $\rho_m\in(0,1)$ such that for all $t\geq0$
$$\sup\limits_{|f|\leq\Tilde{V}_m}|P^m_tf(x,y)-\pi_m(f)|\leq D_m\Tilde{V}_m(x,y)\rho_m^t.$$
\end{proof}

\subsection{Proof of Proposition \ref{egeq2}}

\begin{proof}
 Let $V(x)=1+|x|^2.$  By Assumptions \ref{A1} and \ref{A2}, we have
 \begin{align}
     \mathcal{A}V(x)=&2\langle b(x),x\rangle+||\sigma(x)||^2\leq 2c_1-2c_2|x|^2+||\sigma||^2_\infty\notag\\
     =&-2c_2(1+|x|^2)+2c_1+2c_2+||\sigma||^2_\infty.\label{AVX}
 \end{align}
 From \cite[Theorem 6.1]{MT3}, there exist positive constants $\kappa_1$ and $\kappa_2$ such that 
 \begin{align*}
     \sup\limits_{|f|\leq 1+V}|P_tf(x)-\nu(f)|\leq \kappa_1 V(x)e^{-\kappa_2 t}.
 \end{align*}
 For any 1-Lipschitz function $h$ with $h(0)=0,$ we have 
 \begin{align*}
     |h(x)|\leq|x|\leq1+|x|^2.
 \end{align*}
 Therefore,
 \begin{align*}
     |P_th(x)-\nu(h)|\leq \kappa_1 (1+|x|^2)e^{-\kappa_2 t}.
 \end{align*}
 By  Kantorovich duality, we obtain \eqref{eg1}.
\end{proof}

\section{Proof of Lemmas \ref{XMYMesti} and \ref{MOMX}}\label{se5}

\begin{proof}[Proof of Lemma \ref{XMYMesti}]

    \underline{\emph{Proof of {\rm (1)}}}:
	Recall $(X_0^m, Y_0^m)\sim\pi_m$ with $\pi_m$ being the ergodic measure of the system \eqref{EQ1}. By \eqref{XYbounded},  there exists some positive constant $C$ such that $$ \mathbb{E}V_m(X_t^m,Y_t^m)\leq C,\ \forall t\geq0.$$

By a straightforward computation,  we have for any integer $p\geq2,$
	\begin{equation*}
		\mathcal{A}_mV^p_m(x,y)=p V^{p-1}_m(x, y)\mathcal{A}_mV_m(x,y)+\frac{p(p-1)}{2m^2}V_m^{p-2}(x,y)|\sigma^T(x)(12m^2y+4mx)|^2.
	\end{equation*}
	By Assumption \ref{A1} and \eqref{AMVM} above,  it yields
	\begin{align*}
		\mathcal{A}V_m^p(x,y)\leq&-\frac{c_2 p}{8}V^p_m(x,y)+CV^{p-1}_m(x,y)+
        \frac{p(p-1)}{2m^2}||\sigma||^2_\infty m^2V_m^{p-2}(x,y)|3my+x|^2\\
		\leq&-\frac{c_2 p}{8}V^p_m(x,y)+CV^{p-1}_m(x,y)+CV^{p-2}_m(x,y)(9m^2|y|^2+|x|^2)\\
		\leq&-\frac{c_2 p}{8}V^p_m(x,y)+CV^{p-1}_m(x,y).
	\end{align*}
	Since $(X_t^m, Y_t^m)\sim\pi_m, $ we get
	$$\mathbb{E}[V^p_m(X_t^m, Y_t^m)]\leq C_p\mathbb{E}V^{p-1}_m(X_t^m,Y_t^m).$$
	Therefore, for any $p\in\mathbb{N},$ according to 
	$\mathbb{E}|X_t^m|^{2p}\leq\mathbb{E}[V^p_m(X_t^m, Y_t^m)], $ we get (1).

\underline{\emph{Proof of {\rm (2)}}}:
	We now establish the moment bounds for $|Y_t^m|^2$ and $|Y_t^m|^4$. Recall that under the invariant distribution $\pi_m$, the process $(X_t^m, Y_t^m)$  is stationary, hence all expectations below are independent of $t$. 

Applying It\^o's formula, we have 
	\begin{align*}
		0=d\mathbb{E}|Y_t^m|^2=\frac{2}{m}\mathbb{E}\langle Y_t^m, b(X_t^m)-Y_t^m\rangle dt+\frac{1}{2m^2}\mathbb{E}||\sigma(X_t^m)||^2dt,
	\end{align*}
	which implies
	\begin{align*}
		\mathbb{E}|Y_t^m|^2=\mathbb{E}\langle Y_t^m, b(X_t^m)\rangle+\frac{1}{4m}\mathbb{E}||\sigma(X_t^m)||^2.
	\end{align*}

     For later use we introduce the notation
 $$\gamma_b^p:=\sup_{t\geq0}\mathbb{E}|b(X_t^m)|^p, \ \ \ \ p \in \mathbb{N}.$$ 
 Under Assumption \ref{A0} and the moment bound in (1),  we have $\gamma_b^p<\infty$ and does not depend on $m$.  
Using Young's inequality,  we get
	\begin{equation*}
		\mathbb{E}|Y_t^m|^2\leq\frac{1}{2}\mathbb{E}[|Y_t^m|^2+|b(X_t^m)|^2]+\frac{||\sigma||^2_\infty}{4m},
	\end{equation*}
	which yields
	\begin{equation*}
		\mathbb{E}|Y_t^m|^2\leq\gamma_b^2+\frac{||\sigma||^2_\infty}{2m}.
	\end{equation*}

	The same arguments above apply for the fourth-moment.  By It\^o's formula, we have
	\begin{align*}
		d\mathbb{E}|Y_t^m|^4=\frac{8}{m}\mathbb{E}|Y_t^m|^2\langle Y_t^m, b(X_t^m)-Y_t^m\rangle dt+\frac{1}{m^2}\mathbb{E}[4|\sigma(X_t^m)^TY_t^m|^2+2|Y_t^m|^2||\sigma(X_t^m)||^2]dt.
	\end{align*}
	Since $d\mathbb{E}|Y_t^m|^4=0$ by the stationarity, and by applying Young's inequality, we get
	\begin{align*}
		\mathbb{E}|Y_t^m|^4=&\mathbb{E}|Y_t^m|^2\langle Y_t^m, b(X_t^m)\rangle+\frac{1}{2m}\mathbb{E}|\sigma(X_t^m)^TY_t^m|^2+\frac{1}{4m}\mathbb{E}|Y_t^m|^2
		||\sigma(X_t^m)||^2\\
		\leq&\frac{3}{4}\mathbb{E}|Y_t^m|^4+\frac{1}{4}\mathbb{E}|b(X_t^m)|^4+\frac{||\sigma||^2_\infty}{2m}\mathbb{E}|Y_t^m|^2+\frac{3||\sigma||^2_\infty}{4m}\mathbb{E}|Y_t^m|^2.
	\end{align*}
For $m<1, $ we have
	\begin{align*}
		\mathbb{E}|Y_t^m|^4\leq\gamma_b^4+\frac{3||\sigma||^2_\infty}{m}(\gamma_b^2+\frac{||\sigma||^2_\infty}{2m})\leq\gamma_b^4+\frac{||\sigma||^2_\infty}{2m^2}(\gamma_b^2+||\sigma||^2_\infty).
	\end{align*}
    Similar arguments apply to the case of $p>4$ by induction.
    
\underline{\emph{Proof of {\rm (3)}}}:
	We denote $\delta_t^m:=X_t^m-X_0^m.$ By \cite{F},  the solution of \eqref{EQ1} can be written as 
	\begin{align*}
		X_t^m=X_0^m+mY_0^m(1-e^{-\frac{t}{m}})+\int_0^t(1-e^{-\frac{t-s}{m}})b(X_s^m)ds+\int_0^t(1-e^{-\frac{t-s}{m}})\sigma(X_s^m)dB_s.
	\end{align*}
Then, 
	\begin{align*}
		|\delta_t^m|^p\leq& C[m^p|Y_0^m|^p(1-e^{-\frac{t}{m}})^p+|\int_{0}^{t}(1-e^{-\frac{t-s}{m}})b(X_s^m)ds|^p\\
		+&|\int_0^t\sigma(X_s^m)dB_s|^p+|\int_0^te^{-\frac{t-s}{m}}\sigma(X_s^m)dB_s|^p].
	\end{align*}
	By H\"older's inequality and the finiteness of $\gamma_b^p,$ 
	\begin{align*}
		\mathbb{E}|\int_0^t(1-e^{-\frac{t-s}{m}})b(X_s^m)ds|^p\leq Ct^p.
	\end{align*}
	By BDG's inequality, 
	\begin{align*}
		\mathbb{E}|\int_0^t\sigma(X_s^m)dB_s|^p\leq C\mathbb{E}[\int_{0}^{t}||\sigma(X_s^m)||^2ds]^\frac{p}{2}\leq Ct^\frac{p}{2}, 
	\end{align*}
and
	\begin{align*}
		\mathbb{E}|\int_0^te^{-\frac{t-s}{m}}\sigma(X_s^m)dB_s|^p\leq Cm^\frac{p}{2}.
	\end{align*}
	Noting that $\mathbb{E}|Y_0^m|^p\leq\frac{C}{m^\frac{p}{2}} $ and combining the above estimates together,
	we have
	\begin{equation}\label{aa}
		\mathbb{E}|\delta_t^m|^p\leq C(m^\frac{p}{2}+t^\frac{p}{2}).
	\end{equation}

\underline{\emph{Proof of {\rm (4)}}}: It is sufficient to prove the results when $p$ is even.  Let  $Z:=|\delta_t^m|^p, $ then
	\begin{align*}
\mathbb{E}|\delta_t^m|^p|\ln|\delta_t^m||^p=\frac{1}{p^p}\mathbb{E}Z|\ln Z|^p.
	\end{align*}
	Let $\alpha:=t^\frac{p}{2}+m^\frac{p}{2}, W:=\frac{Z}{\alpha}$. By \eqref{aa}, we have $E[W^2]+\mathbb{E}[W]\leq C.$ 
	Therefore, 
	\begin{align*}
		\mathbb{E} Z|\ln Z|^p\leq\alpha\mathbb{E}[W(|\ln\alpha|+|\ln W|)^p]\leq C\alpha\ln^p\alpha+\alpha\mathbb{E}[W|\ln W|^p].
	\end{align*}
For the second term, 
	\begin{align*}
		\mathbb{E}[W|\ln W|^p]=\mathbb{E}[W|\ln W|^p\mathbb{I}_{W<1}]+\mathbb{E}[W|\ln W|^p\mathbb{I}_{W\geq1}].
	\end{align*}
	Since $f(x)=x(\ln x)^p$ is bounded on $(0, 1), $ we have 
	\begin{align*}
		\mathbb{E}[W|\ln W|^p\mathbb{I}_{W<1}]\leq C.
	\end{align*}
	On the other hand, because $(\ln x)^p\leq C(x+1)$ on $x\geq1, $ we get
	\begin{align*}
		\mathbb{E}[W|\ln W|^p\mathbb{I}_{W\geq1}]\leq C\mathbb{E}[W(W+1)\mathbb{I}_{W\geq1}]\leq C\mathbb{E}[W(W+1)]\leq C.
	\end{align*}
	 Combining these estimates together, we obtain 
	\begin{align*}
		\mathbb{E}|\delta_t^m|^p|\ln|\delta_t^m||^p\leq C\alpha|\ln\alpha|^p+C\alpha=C(t^\frac{p}{2}+m^\frac{p}{2})|\ln(t+m)|^p+C(t^\frac{p}{2}+m^\frac{p}{2}).
	\end{align*}
\end{proof}

\begin{proof}[Proof of Lemma \ref{MOMX}]
	From \eqref{AVX}, we have
	\begin{equation*}
		\mathbb{E}V(X_t)\leq e^{-2c_2t}\mathbb{E}V(X_0)+\frac{2c_1+2c_2+||\sigma||^2_\infty}{2c_2}(1-e^{-2c_2t}).
	\end{equation*}
	If $X_0\sim\nu$, then $\mathbb{E}V(X_t)=\mathbb{E}V(X_0)$ , and hence $$\mathbb{E}V(X_t)\leq\frac{2c_1+2c_2+||\sigma||^2_\infty}{2c_2}.$$
    Since $|x|^2\leq V(x),$ the proof is finished.
\end{proof}

\section{Proof of Proposition \ref{regularity} and \ref{logf}}\label{se66}
\subsection{Notation}

For $\alpha\in(0,1], $  a function $f$ is said to be $\alpha$-H\"older continuous in $\mathbb{R}^d$ if the seminorm $$[f]_\alpha:=\sup\limits_{x\neq y}\frac{|f(x)-f(y)|}{|x-y|^\alpha}$$is finite.  We denote by $C^\alpha(\mathbb{R}^d)$  space of all functions with finite $C^\alpha$-norm. For  a non‑negative integer $k$ and $\alpha\in(0,1],$ the H\"older space $C^{k,\alpha}(\mathbb{R}^d)$ is the subspace of $C^k(\mathbb{R}^d)$ consisting of functions whose $k-$th order partial derivatives are $\alpha-$H\"older continuous. Let $C^{k,\alpha}_b(\mathbb{R}^d)$ be the space of bounded functions in $C^{k,\alpha}$.

We adopt the following notations from \cite{PDE}. Let $\mathcal{D}$ be an open subset of $\mathbb{R}^d.$ For integers $k=0,1,2$, $\alpha\in(0,1],$ $\tau\geq0$, define
\begin{align*}
[f]_{k,0;\mathcal{D}}&=[f]_{k;\mathcal{D}}=\sup\limits_{x\in\mathcal{D}}|\nabla^kf(x)|,\\
[f]_{k,\alpha;\mathcal{D}}&=\sup\limits_{\substack{x,y\in\mathcal{D}, \\x\neq y}}\frac{|\nabla^kf(x)-\nabla^kf(y)|}{|x-y|^\alpha},\\
|f|_{k;\mathcal{D}}&=\sum_{j=0}^k[f]_{j,0;\mathcal{D}},\\
|f|_{k,\alpha;\mathcal{D}}&=|f|_{k;\mathcal{D}}+[f]_{k,\alpha;\mathcal{D}},\\
[f]^{(\tau)}_{k,0;\mathcal{D}}&=[f]^{(\tau)}_{k;\mathcal{D}}=\sup\limits_{x\in\mathcal{D}}d_x^{k+\tau}|\nabla^kf(x)|,\\
[f]^{(\tau)}_{k,\alpha;\mathcal{D}}&=\sup\limits_{\substack{x,y\in\mathcal{D}, \\x\neq y}}d^{k+\alpha+\tau}_{x,y}\frac{\nabla^kf(x)-\nabla^kf(y)}{|x-y|^\alpha},\\
|f|^{(\tau)}_{k;\mathcal{D}}&=\sum_{j=0}^k[f]^{(\tau)}_{j;\mathcal{D}},\\
|f|^{(\tau)}_{k,\alpha;\mathcal{D}}&=|f|^{(\tau)}_{k;\mathcal{D}}+[f]^{(\tau)}_{k,\alpha;\mathcal{D}},
\end{align*}
where $d_x=dist(x,\partial\mathcal{D})$ and $d_{x,y}=\min(d_x,d_y).$ When $\tau=0,$ write $$[f]_{k;\mathcal{D}}^*:=[f]^{(0)}_{k;\mathcal{D}},\ [f]^*_{(k,\alpha;\mathcal{D})}:=[f]^{(0)}_{k,\alpha;\mathcal{D}}, \ |f|^*_{k,\alpha;\mathcal{D}}:=|f|^{(0)}_{k,\alpha;\mathcal{D}}.$$ 
It follows form \cite[(6.11)]{PDE} that
\begin{align}\label{pdeineq}
|fg|^{\tau_1+\tau_2}_{0,\alpha;\mathcal{D}}\leq|f|^{(\tau_1)}_{0,\alpha;\mathcal{D}}|g|^{(\tau_2)}_{0,\alpha;\mathcal{D}} \text{ for } \tau_1, \tau_2\geq0.
\end{align}

The following classical result is quoted from \cite[Theorem 6.2]{PDE}.
\begin{lem}[Schauder interior estimates]\label{Schauder}
Let $\mathcal{D}$ be a open subset of $\mathbb{R}^d,$ and let $u\in C_b^{2,\alpha}(\mathcal{D})$ be a solution in $\mathcal{D}$ of the equation 
\begin{align*}
\langle a(x),\nabla^2u(x)\rangle_{HS}+\langle b(x),\nabla u(x)\rangle=f,
\end{align*}
where $f\in C^{\alpha}(\mathcal{D})$. Assume there exist two positive constants $\kappa$ and $K$ such that 
\begin{align}\label{aij}
a^{ij}\xi_i\xi_j\geq\kappa|\xi|^2,\ \forall x\in\mathcal{D},\ \xi\in\mathbb{R}^d,
\end{align}
and
\begin{align}\label{aijbibound}
|a^{ij}|^{(0)}_{0,\alpha;\mathcal{D}},|b^i|^{(1)}_{0,\alpha;\mathbb{D}}\leq K.
\end{align}
Then there exists some constant $C=C(d,\alpha,\kappa,K)$ such that 
\begin{align*}
|u|^*_{2,\alpha;\mathcal{D}}\leq C(|u|_{0;\mathcal{D}}+|f|^{2}_{0,\alpha;\mathcal{D}}).
\end{align*}
\end{lem}

 \subsection{Proof of Proposition \ref{regularity}}
 \begin{proof}[Proof of Proposition \ref{regularity}]
 We work with the H\"older exponent $\alpha=1.$ For any $x\in\mathbb{R}^d$, define
 $$r(x):=\frac{1}{2(1+|x|)}\leq\frac{1}{2},\ \mathcal{B}_x:=B_{r(x)}(x).$$ 
 We shall apply Lemma \ref{Schauder} with $\mathcal{D}=\mathcal{B}_x.$ Under the Assumptions  \ref{A0},\ref{A1},  in order to use Lemma \ref{Schauder},  we only need to verify $|b|^{(1)}_{0,1;\mathbb{B}_x}\leq K.$ In fact,
 \begin{align*}
|b|^{(1)}_{0,1;\mathcal{B}_x}=|b|^{(1)}_{0;\mathcal{B}_x}+[b]^{(1)}_{0,1;\mathcal{B}_x}=\sup\limits_{y\in\mathcal{B}_x}d_y|b(y)|+\sup\limits_{y,z\in\mathcal{B}_x}d^2_{y,z}\frac{|b(y)-b(z)|}{|y-z|}.
 \end{align*}
 According to definition of $d_y,d_{y,z},$ we have for any $y,z\in\mathcal{B}_x,$
 \begin{align*}
 d_y,d_{y,z}\leq d_x=\frac{1}{2(1+|x|)}\leq\frac{1}{2}.
 \end{align*}
 Therefore,
 \begin{align*}
 d_y|b(y)|\leq\frac{L_b(1+|y|)}{2(1+|x|)}\leq\frac{L_b(1+|x|+|y-x|)}{2(1+|x|)}\leq\frac{L_b(1+|x|+\frac{1}2)}{2(1+|x|)}\leq \frac{3}{4}L_b,
 \end{align*}
 and
 \begin{align*}
 \sup\limits_{y,z\in\mathcal{B}_x}d^2_{y,z}\frac{|b(y)-b(z)|}{|y-z|}\leq\frac{1}{4}L_b,
 \end{align*}
 which implies $|b|^{(1)}_{0,1;\mathcal{B}_x}\leq L_b.$

 According to Lemma \ref{Schauder}, we have
 \begin{align*}
 |f|^*_{2,1;\mathcal{B}_x}\leq C(|f|_{0;\mathcal{B}_x}+|h-\nu(h)|^{(2)}_{0,1;\mathcal{B}_x}),
 \end{align*}
 which yields
 \begin{align*}
 d_x|\nabla f(x)|\leq&[f]^{(0)}_{1;\mathcal{B}_x}\leq C(|f|_{0;\mathcal{B}_x}+|h-\nu(h)|^{(2)}_{0,1;\mathcal{B}_x}),\\
 d_x^2|\nabla^2f(x)|_{0;\mathcal{B}_x}\leq&[f]^{(0)}_{2;\mathcal{B}_x}\leq C(|f|_{0;\mathcal{B}_x}+|h-\nu(h)|^{(2)}_{0,1;\mathcal{B}_x}).
 \end{align*}
 Since $d_x=\frac{1}{2(1+|x|)},$ we have 
 \begin{align*}
 |\nabla f|_{0;\mathcal{B}_x}\leq&C(|f|_{0;\mathcal{B}_x}+|h-\nu(h)|^{(2)}_{0,1;\mathcal{B}_x})(1+|x|),\\
 |\nabla^2f|_{0;\mathcal{B}_x}\leq&C(|f|_{0;\mathcal{B}_x}+|h-\nu(h)|^{(2)}_{0,1;\mathcal{B}_x})(1+|x|)^2.\\
 \end{align*}
 According to Lemma \ref{soluofstein},
 $$|f|_{0;\mathcal{B}_x}=\sup\limits_{y\in B_x}|f(y)|\leq C\sup\limits_{y\in B_x}(1+|y|^2)\leq C\sup\limits_{y\in B_x}(1+|y-x|^2+|x|^2)\leq C(1+|x|^2).$$
Since $||h||_{Lip}\leq1,h(0)=0,$ we have $|h(x)|\leq|x|$ and $\nu(h)\leq C.$ Therefore,
\begin{align*}
|h-\nu(h)|^{(2)}_{0,1;\mathcal{B}_x}\leq C(1+|x|).
\end{align*}
Combining these estimates we obtain for any $x\in\mathbb{R}^d,$
\begin{align*}
|\nabla f(x)|&\leq|\nabla f|_{0;\mathcal{B}_x}\leq C(1+|x|^3),\\
||\nabla^2f(x)||&\leq|\nabla^2f|_{0;\mathcal{B}_x}\leq C(1+|x|^4).
\end{align*}
 \end{proof}

\subsection{Proof of Proposition \ref{logf}}
In order to prove the Proposition \ref{logf}, we consider the following equation: for any $z\in\mathbb{R}^d,\rho\in(0,\frac{1}{2}],$
\begin{equation}\label{simpleequation}
\bar{\mathcal{A}}f(x)=\frac{1}{2}\langle\sigma(x)\sigma^T(x),\nabla^2f(x)\rangle_{HS}=\bar{h},x\in B_{\rho}(z).
\end{equation}
The regularity result of Eq.\eqref{simpleequation} is stated in the following lemma, whose proof is standard, see for instance \cite[Section 2]{DLPDE}.
\begin{lem}\label{keyPDE}
Suppose that $\bar{h}$ in Eq.\eqref{simpleequation} is locally Lipschitz, Let $f\in C^2(B_\rho(z))$ be a classical solution to Eq.\eqref{simpleequation}. Then, there exists a positive constant $C,$ independent of $z,$ such that for all $x,y\in B_{\frac{\rho}{4}}(z)$ and $x\neq y,$
\begin{align*}
\frac{||\nabla^2f(x)-\nabla^2f(y)||}{|x-y||\ln|x-y||}\leq C(|\nabla^2f|_{0;B_\rho(z)}+|f|_{0;B_\rho(z)}+|\bar{h}|_{0,1;B_\rho(z)}).
\end{align*}
\end{lem}

\begin{proof}[Proof of Proposition \ref{logf}]
    Recall the Stein equation:
    \begin{align*}
       \mathcal{A}f(x)= \langle b(x),\nabla f(x)\rangle+\frac{1}{2}\langle \sigma(x)\sigma^T(x),\nabla^2f(x)\rangle_{HS}=h(x)-\nu(h).
    \end{align*}
    Then we have
    \begin{align*}
        \bar{\mathcal{A}}f(x)=h(x)-\nu(h)-\langle b(x),\nabla f(x)\rangle=:\bar{h}(x),
    \end{align*}

Applying Lemma \ref{keyPDE} with this \(\bar{h}\) , it yields that there exists a positive constant $C,$ independent of $z$ such that, for any $x,y\in B_{\frac{\rho}{4}}(z),x\neq y,$
\begin{align*}
    &\frac{||\nabla^2f(x)-\nabla^2f(y)||}{|x-y||\ln|x-y||}\leq C(|\nabla^2f|_{0;B_\rho(z)}+|f|_{0;B_\rho(z)}+|\bar{h}|_{0,1;B_\rho(z)})\\
    \leq&C(|\nabla^2f|_{0;B_\rho(z)}+|f|_{0;B_{\rho}(z)}+|h|_{0,1;B_{\rho}(z)}+|b|_{0,1;B_\rho(z)}|\nabla f|_{0,1;B_\rho(z)}+1),
\end{align*}
where the last inequality follows from \eqref{pdeineq}.

For any $z\in\mathbb{R}^d,$ note that
\begin{align*}
    |b|_{0,1;B_\rho(z)}=\sup\limits_{y\in B_\rho(z)}|b(y)|\leq L_b(1+\sup\limits_{y\in B_\rho(z)}|y|)\leq C(1+|z|),
\end{align*}
and
\begin{align*}
    |h|_{0,1;B_\rho(z)}\leq\rho\leq\frac{1}{2}.
\end{align*}
Consequently,
\begin{align*}
    \frac{||\nabla^2f(x)-\nabla^2f(y)||}{|x-y||\ln|x-y||}\leq C(1+|z|)(|\nabla^2f|_{0;B_\rho(z)}+|f|_{0;B_{\rho}(z)}+|\nabla f|_{0,1;B_\rho(z)}+1).
\end{align*}

From Lemma \ref{soluofstein} we know \(|f(y)|\le C(1+|y|^2)\); therefore
\[
|f|_{0;B_\rho(z)}\le C\bigl(1+\sup_{y\in B_\rho(z)}|y|^2\bigr)
                 \le C(1+|z|^2).
\]
By Proposition \ref{regularity}, \(\|\nabla^2f(y)\|\le C(1+|y|^4)\), so
\[
|\nabla^2f|_{0;B_\rho(z)}\le C\bigl(1+\sup_{y\in B_\rho(z)}|y|^4\bigr)
                          \le C(1+|z|^4).
\]
Finally, using the mean-value theorem, we have
\[
|\nabla f|_{0,1;B_\rho(z)}
=\sup_{\substack{x,y\in B_\rho(z)\\ x\neq y}}
  \frac{|\nabla f(x)-\nabla f(y)|}{|x-y|}
\le |\nabla^2f|_{0;B_\rho(z)}
\le C(1+|z|^4).
\]

Combining these above estimates together, we obtain that for  $z\in\mathbb{R}^d, \rho\in(0,\frac{1}{2}]$ and  any $x,y\in B_{\frac{\rho}{4}}(z),$
\begin{align*}
    \frac{||\nabla^2f(x)-\nabla^2f(y)||}{|x-y||\ln|x-y||}\leq C(1+|z|^5).
\end{align*}
The constant \(C\) does not depend on \(z\).  
In particular, for any fixed \(x\in\mathbb{R}^d\) we may choose \(z=x\) and \(\rho=\frac12\).  
Then \(B_{1/8}(x)\subset B_{\rho/4}(x)\), and for every \(y\) with \(0<|y-x|\le\frac18\) we obtain
\[
\frac{\|\nabla^2f(x)-\nabla^2f(y)\|}{|x-y|\,|\ln|x-y||}\le C(1+|x|^5).
\]
Taking the supremum over such \(y\) completes the proof of Proposition \ref{logf}.

\end{proof}

\bibliographystyle{plain}

\begin{thebibliography}{1}

\bibitem{fbm1} 
B. Boufoussi and C. A. Tudor, \textit{Kramers-Smoluchowski approximation for stochastic evolution equations with FBM}, Rev. Roumaine Math. Pures Appl. \textbf{50} (2005), no. 2, 125--136.

\bibitem{DNB} 
N. Breimhorst, \textit{Smoluchowski-Kramers Approximation for Stochastic Differential Equations with non-Lipschitzian coefficients}, PhD Thesis, Fakult\"at f\"ur Mathematik, Universit\"at Bielefeld, 2009.

\bibitem{SPDE1} 
S. Cerrai and M. Freidlin, \textit{Smoluchowski-Kramers approximation for a general class of SPDEs}, J. Evol. Equ. \textbf{6} (2006), no. 4, 657--689.

\bibitem{SPDE2} 
S. Cerrai and M. Freidlin, \textit{On the Smoluchowski-Kramers approximation for a system with an infinite number of degrees of freedom}, Probab. Theory Related Fields \textbf{135} (2006), no. 3, 363--394.

\bibitem{SMF} 
S. Cerrai, M. Freidlin and M. Salins, \textit{On the Smoluchowski-Kramers approximation for SPDEs and its interplay with large deviations and long time behavior}, Discrete Contin. Dyn. Syst. \textbf{37} (2017), no. 1, 33--76.

\bibitem{JFA} 
S. Cerrai and N. Glatt-Holtz, \textit{On the convergence of stationary solutions in the Smoluchowski-Kramers approximation of infinite dimensional systems}, J. Funct. Anal. \textbf{278} (2020), no. 8, 108421.

\bibitem{LDP1} 
S. Cerrai and M. Salins, \textit{Smoluchowski-Kramers approximation and large deviations for infinite dimensional gradient systems}, Asymptot. Anal. \textbf{88} (2014), no. 4, 201--215.


\bibitem{stein0} 
S. Chatterjee and E. Meckes, \textit{Multivariate normal approximation using exchangeable pairs}, ALEA Lat. Am. J. Probab. Math. Stat. \textbf{4} (2008), 257--283.

\bibitem{C} 
Z. Chen and M. Freidlin, \textit{Smoluchowski-Kramers approximation and exit problems}, Stoch. Dyn. \textbf{5} (2005), no. 4, 569--585.

\bibitem{PDE} 
D. Gilbarg and N. S. Trudinger, \textit{Elliptic Partial Differential Equations of Second Order}, Springer, Berlin, 1977.

\bibitem{DeTe2023}
P. S. Dey and G. Terlov, \textit{Stein's method for conditional central limit theorem}, Ann. Probab. \textbf{51} (2023), no. 2, 723--773.

\bibitem{DLPDE} 
K. Du and J. Liu, \textit{On the Cauchy problem for stochastic parabolic equations in H\"older spaces}, Trans. Amer. Math. Soc. \textbf{371} (2019), no. 4, 2643--2664.

\bibitem{EGZ2019}
A. Eberle, A. Guillin and R. Zimmer, \textit{Couplings and quantitative contraction rates for Langevin dynamics}, Ann. Probab. \textbf{47} (2019), 1982--2010.

\bibitem{FaKo2021}
X. Fang and Y. Koike, \textit{High-dimensional central limit theorems by Stein's method}, Ann. Appl. Probab. \textbf{31} (2021), no. 4, 1660--1686.

\bibitem{FSX} 
X. Fang, Q.-M. Shao and L. Xu, \textit{Multivariate approximations in Wasserstein distance by Stein's method and Bismut's formula}, Probab. Theory Related Fields \textbf{174} (2019), no. 3-4, 945--979.

\bibitem{F} 
M. Freidlin, \textit{Some remarks on the Smoluchowski-Kramers approximation}, J. Stat. Phys. \textbf{117} (2004), no. 3-4, 617--634.

\bibitem{stein1} 
L. Goldstein and Y. Rinott, \textit{Multivariate normal approximations by Stein's method and size bias couplings}, J. Appl. Probab. \textbf{33} (1996), no. 1, 1--17.

\bibitem{GLWZ} A. Guillin, W. Liu, L. Wu and C. Zhang, \textit{The kinetic Fokker-Planck equation with mean field interaction},  J. Math. Pures Appl. \textbf{150} (2021), 1-23.

\bibitem{GMZ2021}
A. Guillin and P. Monmarch\'e, \textit{Uniform long-time and propagation of chaos estimates for mean field kinetic particles in non-convex landscapes}, J. Stat. Phys. \textbf{185} (2021), Article 15.

\bibitem{GWXZ2026}
A. Guillin, Y. Wang, L. Xu and H. Yang, \textit{Error estimates between SGD with momentum and underdamped Langevin diffusion}, Ann. Appl. Probab., (2026+), in press.  


\bibitem{GuWa2012} 
A. Guillin and F.-Y. Wang, \textit{Degenerate Fokker--Planck equations: Bismut formula, gradient estimate and Harnack inequality}, J. Differential Equations \textbf{253} (2012), 20--40.


\bibitem{HDC} 
Z. He, J. Duan and X. Cheng, \textit{A parameter estimator based on Smoluchowski-Kramers approximation}, Appl. Math. Lett. \textbf{90} (2019), 54--60.

\bibitem{HMVW} 
S. Hottovy, A. McDaniel, G. Volpe and J. Wehr, \textit{The Smoluchowski-Kramers limit of stochastic differential equations with arbitrary state-dependent friction}, Comm. Math. Phys. \textbf{336} (2015), no. 3, 1259--1283.

\bibitem{fbm3} 
W. Liu, B. Pei and Q. Yu, \textit{Rate of convergence for the Smoluchowski-Kramers approximation for distribution-dependent SDEs driven by fractional Brownian motions}, Stoch. Dyn. \textbf{24} (2024), no. 1, Paper No. 2450002.

\bibitem{stein2} 
L. Mackey and J. Gorham, \textit{Multivariate Stein factors for a class of strongly log-concave distributions}, Electron. Commun. Probab. \textbf{21} (2016), Paper No. 56.

\bibitem{MT3} 
S. P. Meyn and R. L. Tweedie, \textit{Stability of Markovian processes. III. Foster-Lyapunov criteria for continuous-time processes}, Adv. in Appl. Probab. \textbf{25} (1993), no. 3, 518--548.

\bibitem{NK1} 
K. Narita, \textit{Asymptotic behavior of velocity process in the Smoluchowski-Kramers approximation for stochastic differential equations}, Adv. in Appl. Probab. \textbf{23} (1991), no. 2, 317--326.

\bibitem{NK2} 
K. Narita, \textit{The Smoluchowski-Kramers approximation for the stochastic Li\'enard equation by mean-field}, Adv. in Appl. Probab. \textbf{23} (1991), no. 2, 303--316.

\bibitem{NK3} 
K. Narita, \textit{Asymptotic behavior of fluctuation and deviation from limit system in the Smoluchowski-Kramers approximation for SDE}, Yokohama Math. J. \textbf{42} (1994), no. 1, 41--76.

\bibitem{N} 
E. Nelson, \textit{Dynamical Theories of Brownian Motion}, Princeton University Press, Princeton, 1967.

\bibitem{SPDE4} 
H. D. Nguyen, \textit{The small-mass limit and white-noise limit of an infinite dimensional generalized Langevin equation}, J. Stat. Phys. \textbf{173} (2018), no. 2, 411--437.

\bibitem{SPDE6} 
H. D. Nguyen, \textit{The small mass limit for long time statistics of a stochastic nonlinear damped wave equation}, J. Differential Equations \textbf{371} (2023), 481--548.

\bibitem{SPDE5} 
M. Salins, \textit{Smoluchowski-Kramers approximation for the damped stochastic wave equation with multiplicative noise in any spatial dimension}, Stoch. Partial Differ. Equ. Anal. Comput. \textbf{7} (2019), no. 1, 86--122.

\bibitem{stein4} 
Q.-M. Shao and Z.-S. Zhang, \textit{Identifying the limiting distribution by a general approach of Stein's method}, Sci. China Math. \textbf{59} (2016), no. 12, 2379--2392.

\bibitem{fbm2} 
T. C. Son, \textit{The rate of convergence for the Smoluchowski-Kramers approximation for stochastic differential equations with FBM}, J. Stat. Phys. \textbf{181} (2020), no. 5, 1730--1745.

\bibitem{BRJ1} 
T. C. Son, D. Q. Le and M. H. Duong, \textit{Rate of convergence in the Smoluchowski-Kramers approximation for mean-field stochastic differential equations}, Potential Anal. \textbf{60} (2024), no. 3, 1031--1065.

\bibitem{Stein1972}
C. Stein, \textit{A bound for the error in the normal approximation to the distribution of a sum of dependent random variables}, in: Proc. Sixth Berkeley Symp. Math. Statist. Probab., Vol. 2, Univ. of California Press, 1972, pp. 583--602.

\bibitem{stein} 
C. Stein, \textit{Approximate Computation of Expectations}, Lecture Notes-Monograph Series 7, Institute of Mathematical Statistics, Hayward, CA, 1986.

\bibitem{BRJ} 
N. V. Tan and N. T. Dung, \textit{A Berry-Esseen bound in the Smoluchowski-Kramers approximation}, J. Stat. Phys. \textbf{179} (2020), no. 4, 871--884.

\bibitem{Wang2013}
F.-c出                                                 Y. Wang, \textit{Harnack Inequalities for Stochastic Partial Differential Equations}, Springer Monographs in Mathematics, Springer, New York, 2013.

\bibitem{WaZh13}
F.-Y. Wang and X. Zhang, \textit{Derivative formula and applications for degenerate diffusion semigroups}, J. Math. Pures Appl. \textbf{99} (2013), 726--740.

\bibitem{Wu2001} 
L. Wu, \textit{Large and moderate deviations and exponential convergence for stochastic damping Hamiltonian systems}, Stochastic Process. Appl. \textbf{91} (2001), 205--238.

\bibitem{XY} 
L. Xie and L. Yang, \textit{The Smoluchowski-Kramers limits of stochastic differential equations with irregular coefficients}, Stochastic Process. Appl. \textbf{150} (2022), 91--115.

\end{thebibliography}

\end{document}